\documentclass{amsart}

\usepackage{amsmath,amssymb,amsthm}

\theoremstyle{plain}
\newtheorem{thm}{Theorem}[section]
\newtheorem{prop}[thm]{Proposition}
\newtheorem{lem}[thm]{Lemma}
\newtheorem{cor}[thm]{Corollary}

\theoremstyle{definition}
\newtheorem{example}[thm]{Example}
\newtheorem{remark}[thm]{Remark}

\newcommand{\sH}{\mathcal{ H}}

\newcommand{\sM}{\mathcal{ M}}

\newcommand{\sV}{\mathcal{ V}}

\newcommand{\sW}{\mathcal{ W}}

\newcommand{\sN}{\mathcal{ N}}

\def\d{\delta}

\def\la{\Lambda}

\def\om{\Omega}

\def\tht{\Theta}

\def\ts{\times}

\def\iy{\infty}
\def\im{{\rm Im\, }}

\def\kr{{\rm Ker\, }}

\def\rank{{\rm rank\, }}

\def\lg{\langle}
\def\rg{\rangle}

\def\BC{{\mathbb C}}
\def\BD{{\mathbb D}}

\def\BT{{\mathbb T}}

\newcommand{\mat}[2]{\ensuremath{\left[\begin{array}{#1}#2\end{array} \right]}}
\newcommand{\wtil}[1]{\widetilde{#1}}
\newcommand{\sbm}[1]{\left[\begin{smallmatrix} #1\end{smallmatrix}\right]}
\newcommand{\degM}{\delta}
\newcommand{\half}{\frac{1}{2}}
\newcommand{\fR}{\mathfrak{R}}
\newcommand{\fF}{\mathfrak{F}}

\newcommand{\ands}{\quad\mbox{and}\quad}

\begin{document}

\title[Rational matrix solutions to the Leech equation]{Rational matrix solutions to the Leech equation: The Ball-Trent approach revisited}

\author{Sanne ter Horst}
\address{School of Computer, Statistical and Mathematical Sciences\\
North-West University\\
Potchefstroom 2520\\ South Africa}
\email{sanne.terhorst@nwu.ac.za}

\keywords{Leech equation, Toeplitz operators, stable rational matrix functions, outer spectral factor}

\begin{abstract}
Using spectral factorization techniques, a method is given by which rational matrix solutions to the Leech equation with rational matrix data can be computed explicitly.
This method is based on an approach by J.A. Ball and T.T. Trent, and generalizes techniques from recent work of T.T. Trent for the case of polynomial matrix data.
\end{abstract}

\maketitle

\setcounter{section}{-1}
\setcounter{equation}{0}
\section{Introduction}\label{intro}

Consider $H^\iy$-matrix functions $G\in H_{m\ts p}^\iy $ and $K\in H_{m\ts q}^\iy$, and let $T_G:\ell^2_+(\BC^p)\to\ell^2_+(\BC^m)$ and $T_K:\ell^2_+(\BC^q)\to\ell^2_+(\BC^m)$ be the corresponding (block) Toeplitz operators. See Section \ref{S:prelim} below for the definitions of these spaces and operators. A beautiful unpublished result of R.B. Leech (cf., \cite{RR85}) tells us that there exists an $X\in H_{p\ts q}^\iy $ such that
\begin{equation}\label{Leech1}
G(z)X(z)=K(z)\quad (z\in\BD), \ands \|X\|_\iy\leq 1,
\end{equation}
with $\BD$ the open unit disc in $\BC$, if and only if
\begin{equation}\label{poscond1}
T_GT_G^*-T_KT_K^* \ \mbox{ is positive}.
\end{equation}
Note that \eqref{Leech1} is equivalent to $T_GT_X=T_K$ and $\|T_X\|\leq 1$. Hence Leech's theorem can be viewed as the analogue of the Douglas factorization lemma \cite{D66} within the class of analytic Toeplitz operators. The necessity of \eqref{poscond1} follows directly from Douglas' factorization lemma and the reformulation of \eqref{Leech1} in terms of Toeplitz operators. The other implication is more involved. The solution criterion \eqref{poscond1} can also be formulated directly in terms of the functions $G$ and $K$, it is equivalent to the map
\begin{equation}\label{poskern}
L(z,w)=\frac{G(z)G(w)^*-K(z)K(w)^*}{1-z\overline{w}}\quad(z,w\in\BD)
\end{equation}
being a positive kernel in the sense of Aronszajn \cite{A50}, that is, for any finite sequence $z_1,\ldots,z_n\in\BD$ the block operator matrix $[L(z_i,z_j)]_{i,j=1,\ldots,n}$ defines a positive operator on the Hilbert space direct sum of $n$ copies of $\BC^m$. We note that the actual result by Leech is stated in the general context of Hilbert space operators intertwining shift operators, and in particular holds for operator-valued $H^\iy$-functions as well. Our interest is primarily in the case where $G$ and $K$ are rational matrix functions. 

There exists various proofs of  Leech's theorem, see \cite{Nikol86} and the references therein. In \cite{BT98} Ball and Trent prove a generalization of Leech's theorem to the polydisc in $\BC^d$, adapting a technique coined the `lurking isometry' approach in \cite{B00}, and give a description of all $X\in H_{p\ts q}^\iy $ satisfying \eqref{Leech1}.
We briefly outline the construction here, specified to the single variable case.

The positivity of $T_GT_G^*-T_KT_K^*$ implies we can factor $T_GT_G^*-T_KT_K^*$ as $T_GT_G^*-T_KT_K^*=\la_\circ \la_\circ^*$, for some operator $\la_\circ:\sH_\circ\to\ell^2_+(\BC^m)$ such that $\kr\la_\circ=\{0\}$. The latter implies that $\dim\sH_\circ=\rank (T_GT_G^*-T_KT_K^*)$. Such a factorization is often referred to as a Kolmogorov decomposition in the literature, cf., \cite{C96}. Let $\hat{\la}_\circ$ be the analytic operator-valued function on $\BD$, with values $\hat{\la}_\circ(z):\sH_\circ\to\BC^m$, $z\in\BD$, defined by $\hat{\la}_\circ(z)h=(\fF_m\la_\circ h)(z)$, $z\in\BD,\, h\in\sH_\circ$. Here $\fF_m$ is the Fourier transform mapping $\ell^2_+(\BC^m)$ isometrically onto the Hardy space $H^2_m$. Next one verifies that $G$, $K$ and $\hat{\la}_\circ$ satisfy the following identity:
\begin{align}
 &z\overline{w}\hat{\la}_\circ(z)\hat{\la}_\circ(w)^*+G(z)G(w)^*=\nonumber  \\
&\hspace{3cm}  =\hat{\la}_\circ(z)\hat{\la}_\circ(w)^*+K(z)K(w)^*\quad (z,w\in\BD).\label{fundid1}
\end{align}
From this identity one derives the existence of a partial isometry
\begin{equation}\label{defMintro}
M_\circ=\begin{bmatrix}A_\circ & B_\circ\\ C_\circ & D_\circ
\end{bmatrix}:\begin{bmatrix}\sH_\circ\\ \BC^q\end{bmatrix}
\to\begin{bmatrix}\sH_\circ\\ \BC^p\end{bmatrix}
\end{equation}
such that
\begin{equation}\label{contracon2intro}
\mat{cc}{z \hat{\la}_\circ(z) & G(z)}M_\circ
=\mat{cc}{\hat{\la}_\circ(z) & K(z)}\quad (z\in\BD).
\end{equation}
This in turn implies that the function $X$ defined on $\BD$ by
\begin{equation}\label{sol0}
X(z)=D_\circ+zC_\circ(I-zA_\circ)^{-1}B_\circ\quad (z\in\BD)
\end{equation}
is in $H^\iy_{p\ts q}$ and satisfies \eqref{Leech1}. If one considers all contractions $M_\circ$ of the form \eqref{defMintro} such that \eqref{contracon2intro} holds, possibly enlarging $\sH_\circ$, all solutions $X$ to \eqref{Leech1} are obtained via \eqref{sol0}.

From the point of view of rational matrix functions the above construction has one disadvantage.  In general,  the Hilbert space $\sH_\circ$ appearing in \eqref{defMintro}  is infinite dimensional, and in that case it is hard to see when the solution $X$ in \eqref{sol0} is rational. In fact, even if  both $G$ and $K$ are rational matrix functions, $T_GT_G^*-T_KT_K^*$ may very well be of infinite rank. More precisely,   see Theorem \ref{T:rational} below,  in the rational matrix case  $T_GT_G^*-T_KT_K^*$ has finite rank if and only if $G(e^{it})G(e^{it})^*=K(e^{it})K(e^{it})^*$ for all $0\leq t \leq 2\pi$. Overcoming this difficulty is the main theme of the present paper.

In the context of the Toeplitz-corona problem, which can be reduced to the special case of \eqref{Leech1} with $q=m$ and $K(z)=I_m$, $z\in\BD$, Trent \cite{Trent07} deduced a modification  of the above procedure for the special case that $G$ is a row vector ($m=1$) polynomial, leading to a rational column vector solution of McMillan degree at most the highest degree of the polynomials occurring in $G$. Throughout this paper, the McMillan degree of a rational matrix function $V$ will be denoted by $\d(V)$; see Section \ref{S:prelim} for the precise definition of $\d(V)$. The procedure of \cite{Trent07} was recently extended in \cite{Trent13} to the general case of the Leech equation \eqref{Leech1}, with $G$ and $K$ rational matrix functions, by reducing it to the case where $G$ and $K$ have polynomial entries, and solving the latter problem via techniques similar to those in \cite{Trent07}.

In the present paper we also consider the Leech equation \eqref{Leech1} with $G$ and $K$ rational matrix functions. However, instead of reducing to the case of polynomial data, we associate our problem with another Leech equation, with data functions $G$ and $\wtil{K}$, i.e., with the same $G$. The advantage of our approach is that we keep better track of the McMillan degrees in our computations, leading to sharper bounds on the McMillan degrees of the solutions. The construction of $\wtil{K}$ even works in the case where $G$ and $K$ are not rational, provided that the function $R\in L^\iy_{m\ts m}$ defined by
\begin{equation}\label{defR}
R(e^{it})=G(e^{it})G(e^{it})^*-K(e^{it})K(e^{it})^*\quad (a.e.\ t\in[0,2\pi])
\end{equation}
admits an outer spectral factor, that is, a function $\Phi\in H^\iy_{r\ts m}$, for some $r\leq m$, with $T_R=T_\Phi^* T_\Phi$ and $\ker T_\Phi^*=\{0\}$. Note that outer spectral factors are unique up to multiplication with a unitary constant matrix on the left, hence, with some abuse of terminology, we will refer to the outer spectral factor, provided it exist. If $G$ and $K$ are rational, then so is $R$, and this implies an outer spectral factor of $R$ exists.

Our method requires the following procedure:
\begin{itemize}
\item[1.]
Define $R\in L^\iy_{m\ts m}$ by \eqref{defR}. Then $T_R$ is positive. Assume $R$ admits an outer spectral factor $\Phi\in H^\iy_{r\ts m}$, for some $r\leq m$.

\item[2.]
The subspace
\begin{equation}\label{sMGK}
\sM_{\Phi}:=\{f\in \ell_+^2(\BC^r)\mid T_\Phi^*f\in  \overline{\im H_G + \im H_K}\,\}.
\end{equation}
is invariant under the backward shift on $\ell^2_+(\BC^r)$, and hence, by the Beurling-Lax theorem, there exists an inner function $\tht\in H^\iy_{r\ts k}$, for some $k\leq r$, such that the range of $T_\tht$ is the orthogonal complement of $\sM_{\Phi}$.

\item[3.]
Define $F\in L^\iy_{m\ts k}$ by $F(e^{it})=\Phi(e^{it})^*\tht(e^{it})$, for a.e.\ $t\in[0,2\pi]$. Then $F\in H^\iy_{m\ts k}$.
\end{itemize}
The claims in the above steps will be proved Section \ref{S:NewApproach}. The function $F$ defined in Step 3 can be taken as a particular choice for the function $F$ appearing in the next theorem. This theorem provides the basis for our method and is the main result of the present paper; a proof will be given in Section \ref{S:NewApproach}.

\begin{thm}\label{T:main}
Assume $G\in H^\iy_{m\ts p}$ and $K\in H^\iy_{m\ts q}$ such that $T_GT_G^*-T_KT_K^*$ is positive and the function $R$ defined in \eqref{defR} admits an outer spectral factor. Then there exists a function $F\in H^\iy_{m\ts k}$, for some $k\leq m$, such that:
\begin{itemize}
\item[\textup{(i)}] $T_GT_G^*-T_KT_K^*-T_FT_F^*$ is positive;

\item[\textup{(ii)}] $\rank \big(T_GT_G^*-T_KT_K^*-T_FT_F^*\big )\leq \dim \Big(\overline{\im H_G + \im H_K}\Big)$.
\end{itemize}
Here $H_G$ and $H_K$ denote the Hankel operators of $G$ and $K$, respectively.
\end{thm}

Given $F$ as in Theorem \ref{T:main}, we apply the Ball-Trent approach with $K$ replaced by $\wtil{K}=[\,K\ F\,]$. This yields  $H^\iy$-solutions $\wtil{X}=[\, X\ Y\,]$ of
\begin{align}
G(z)\begin{bmatrix} X(z)&\!\!\! Y(z) \end{bmatrix}=\begin{bmatrix} K(z)&\!\!\! F(z) \end{bmatrix}\quad (|z|<1),\ \ \mbox{and}\  \  \|\begin{bmatrix} X&\!\!\! Y \end{bmatrix}\|_\iy\leq 1.  \label{Leech2}
\end{align}
Note that \eqref{Leech2} implies that $X$ satisfies \eqref{Leech1}. Whether or not all solutions of \eqref{Leech1} can be obtained via this procedure is still an open problem.

This procedure is specifically of interest in case $G$ and $K$ are rational matrix functions. In that case the upper bound in (ii) is finite, and serves as an upper bound on the least possible McMillan degree of solutions $\wtil{X}$ to \eqref{Leech2}, hence the same upper bound applies to $X$. The following theorem provides some additional results for the case of rational data functions; a proof will be given in Section \ref{S:RKrat}.

\begin{thm}\label{T:mainrat}
Let $G\in H^\iy_{m\ts p}$ and $K\in H^\iy_{m\ts q}$ be rational matrix functions such that $T_GT_G^*-T_KT_K^*$ is positive. Then the function $R\in L^\iy_{m\ts m}$ defined by \eqref{defR} admits an outer spectral factor $\Phi$. Moreover, in this case the functions $R$, $\Phi$, $\tht$ and $F$ defined in the above procedure are all rational matrix functions whose McMillan degrees satisfy
\begin{equation}\label{MMbounds}
\half\degM(R)=\degM(\Phi)\leq\degM(\tht)=\degM(F)=\dim\sM_{\Phi}<\iy,
\end{equation}
and $\tht$ is two-sided inner, i.e., $k=r$ and $\tht(e^{it})^*\tht(e^{it})=I_r =\tht(e^{it})\tht(e^{it})^*$ for each $t\in[0,2\pi]$. Finally, we have
\begin{equation}\label{TvsH}
T_GT_G^*-T_KT_K^*-T_FT_F^*=H_KH_K^*+H_FH_F^*-H_GH_G^*.
\end{equation}
In particular, the left hand side in inequality (ii) in Theorem \ref{T:main} is equal to $\rank(H_KH_K^*+H_FH_F^*-H_GH_G^*)$.
\end{thm}

Thus, in case $G$ and $K$ are rational matrix functions, the problem reduces to computing a Kolmogorov decomposition of the right hand side of \eqref{TvsH}. Note that there are effective ways to computing Kolmogorov decompositions, cf., \cite{C96}. Moreover, the functions $R$, $\Phi$, $\tht$ and $F$ can be computed explicitly using state space techniques from mathematical systems theory (cf., \cite{BGKR08,FB10}), starting from a state space representation of the function $[\,G\ K\,]$. This will be the topic of a forthcoming paper of the present author together with A.E. Frazho and M.A. Kaashoek.

The paper consists of \ref{S:appendix} sections, not counting the present introduction. Section \ref{S:prelim} contains some of the notations and terminology as well as some operator theory preliminaries used in the sequel. The main result, Theorem \ref{T:main}, is proved in Section \ref{S:NewApproach}. In Section \ref{S:RKrat} the focus lays on the case that $G$ and $K$ are rational matrix functions; a proof of Theorem \ref{T:mainrat} will be given as well as a criterion for the case that $T_GT_G^*-T_KT_K^*$ has finite rank. The final section contains some general operator theoretical results, and their proofs, that are used in the preceding sections.

\setcounter{equation}{0}
\section{Preliminaries}\label{S:prelim}
In this section we introduce notations and terminology used throughout the paper and we present some operator theory preliminaries.

With {\em operator} we mean a continuous linear map acting between two Hilbert spaces. In particular, all operators in this paper are by definition bounded. Invertibility of an operator  means the operator has a bounded inverse. Let $\sH$ be a Hilbert space. A {\em subspace} of $\sH$ is a closed linear manifold within $\sH$. The identity operator on $\sH$ will denoted by $I_\sH$ and the $k\ts k$ identity matrix  by $I_k$. Often these subscripts $\sH$ and $k$ will be omitted.
We say that an operator $T$ on $\sH$ is \emph{positive} whenever the inner product $\lg Tu,u\rg\geq 0$ for each $u\in \sH$, and   $T$ is said to be  \emph{positive definite} whenever $T$ is both positive and invertible. The notations $T\geq0$ and $T>0$ will be used to indicate the positivity, respectively positive definiteness, of $T$. In case $T_1$ and $T_2$ are selfadjoint operators on $\sH$, we will write $T_1\geq T_2$, resp.\ $T_1>T_2$, to indicate $T_1-T_2\geq0$, resp.\ $T_1-T_2>0$.

The symbol $H_{m\ts p}^\iy $ will indicate the Hardy space of all uniformly bounded analytic ${m\ts p}$ matrix-valued functions in the open unit disc. For any $V\in H_{m\ts p}^\iy$ the supremum norm of $V$ is defined by $\|V\|_\iy=\sup_{|z|<1}\|V(z)\|$, making $H_{m\ts p}^\iy$ into a Banach space. Here we follow the convention that the norm $\|M\|$ of an $m\ts p$ matrix $M$ is equal to the norm of the operator from $\BC^p$ into $\BC^m$ induced by $M$ in the canonical way. We write $L_{m\ts p}^\infty$ for the Banach space consisting of all Lebesgue measurable, essentially bounded ${m\ts p}$-matrix functions on the unit circle $\BT$ together with the essential supremum norm, also denoted by $\|\ \|_\infty$. The space $H_{m\ts p}^\iy$ will be viewed both as a sub-Banach space of $L_{m\ts p}^\infty$ and as a Banach space in its own right.

With a function $Z\in  L_{m\ts p}^\iy$ we associate the functions $Z^*\in L^\iy_{p\ts m}$ and $Z^t\in L^\iy_{m \ts p}$ defined by
\begin{equation}\label{F*Ftil}
Z^*(e^{it})=Z(e^{it})^*\quad\mbox{and}\quad Z^t(e^{it})=Z(e^{-it}) \quad (\textup{a.e.\ } t\in[0,2\pi])
\end{equation}
For $V\in H^\iy_{m \ts p}$, the functions $V^*$ and $V^t$ can be uniquely extended to bounded analytic functions on the open exterior disc $\BC\backslash \overline{\BD}$, infinity included, via the formulas $V^*(z)=V(1/\bar{z})^*$ and $V^t(z)=V(1/z)$, $|z|>1$.

By $\ell^2(\BC^k)$ and $\ell^2_+(\BC^k)$ we denote the Hilbert spaces consisting of bilateral, respectively unilateral, square summable sequences with values in $\BC^k$. Viewing $\ell^2_+(\BC^k)$ as a sub-Hilbert space of $\ell^2(\BC^k)$, we write $\ell^2_-(\BC^k)$ for the orthogonal complement of $\ell^2_+(\BC^k)$ in $\ell^2(\BC^k)$.
The symbol $S_k$ stands for the (block) forward shift on $\ell^2_+(\BC^k)$, and $E_k$ denotes the canonical embedding of $\BC^k$ into $\ell^2_+(\BC^k)$ defined by $E_ku=\begin{bmatrix} u&0&0&\cdots\end{bmatrix}^\top$. Note that $I-S_kS_k^*=E_kE_k^*$.

Let $Z$ be a function in $L_{m\ts p}^\iy$ and denote the Fourier coefficients of $Z$ by $\ldots,Z_{-1},Z_0,Z_1,Z_2,\ldots$. Then we define the (block) {\em Toeplitz operator} $T_Z$ and (block) {\em Hankel operators} $H_{Z,+}$ and $H_{Z,-}$ associated with $Z$ by the operators mapping $\ell^2_+(\BC^p)$ into $\ell^2_+(\BC^m)$ given by their infinite block matrix representations
\begin{equation*}
T_Z =\sbm{
Z_0     & Z_{-1}     & Z_{-2}    & \cdots \\
Z_1     & Z_0        & Z_{-1}    & \cdots \\
Z_2     & Z_1        & Z_0       & \cdots \\
\vdots  & \vdots     & \vdots    & \ddots
},\
H_{Z,+} =\sbm{
Z_1     & Z_2    & Z_3    & \cdots \\
Z_2     & Z_3        & Z_4    & \cdots \\
Z_3     & Z_4        & Z_5       & \cdots \\
\vdots  & \vdots     & \vdots    & \ddots
},\
H_{Z,-} =\sbm{
Z_{-1}     & Z_{-2}    & Z_{-3}    & \cdots \\
Z_{-2}     & Z_{-3}    & Z_{-4}    & \cdots \\
Z_{-3}     & Z_{-4}    & Z_{-5}       & \cdots \\
\vdots  & \vdots     & \vdots    & \ddots
}.
\end{equation*}
We shall refer to $H_{Z,+}$ and $H_{Z,-}$ as the {\em analytic}, respectively {\em anti-analytic}, Hankel operator associated with $Z$. Note that $T_{Z^*}=T_Z^*$ and $H_{Z,+}^*=H_{Z^*,-}$. For $V\in H^\iy_{m\ts p}$ we have $H_{V,-}=0$, and we will simply write $H_V$ for $H_{V,+}$.

Now consider $U\in H^\iy_{n \ts p}$, $V\in H^\iy_{m \ts p}$ and $W\in H^\iy_{m \ts q}$. Then
the following useful identities apply (cf., \cite[Proposition 2.14]{BS90}):
\begin{equation}\label{UsefulIDs}
\begin{aligned}
&T_{V^*W}=T_{V}^*T_W,\quad
T_{UV^*}=T_U T_V^*+H_U H_V^*,\\
&H_{V^*W,+}=T_V^*H_W,\quad
H_{UV^*,+}=H_U T_{V^t}^*.
\end{aligned}
\end{equation}

The sets of rational matrix $L^\iy_{p \ts m}$- and $H^\iy_{p \ts m}$-functions will be denoted by $\fR L^\iy_{m\ts p}$ and $\fR H^\iy_{m\ts p}$, respectively. For a $m\ts p$ rational matrix function $Z$ the {\em McMillan degree} is denoted by $\degM(Z)$ and equals the sum of the local degrees, $\degM(Z)=\sum_{w\in\BC}\degM(Z,w)$. Here the {\em local degree} $\degM(Z,w)$ of $Z$ at $w$ is defined to be the rank of the Hankel operator defined by the negative Fourier coefficients of the Fourier expansion if $Z$ in a deleted neighborhood of $w$. See Section 8.4 in \cite{BGKR08} for more details. It is well known that for $V\in\fR H^\iy_{m\ts p}$ the MacMillan degree $\degM(V)$ equals the rank of the Hankel operator $H_V$. Moreover, for $Z\in\fR L^\iy_{m\ts p}$ the MacMillan degree $\degM(Z)$ is equal to $\rank (H_{Z,+})+\rank (H_{Z,-})$.

\setcounter{equation}{0}
\section{Proof of Theorem \ref{T:main}}\label{S:NewApproach}

Let $G\in H_{m\ts p}^\iy$ and $K\in H_{m\ts q}^\iy$, and define $R\in L^\iy_{m\ts m}$ by \eqref{defR}. Throughout this section we shall assume that $T_GT_G^*\geq T_KT_K^*$. This implies that $R$ is positive on $\BT$. Indeed, note that the positivity of the kernel $L$ in \eqref{poskern} implies that $(1-|z|^2)L(z,z)=G(z)G(z)^*-K(z)K(z)^*$ is positive for each $z\in\BD$. Hence the same is true for the non-tangential limits of $(1-|z|^2)L(z,z)$ to the unit circle, which exist for  almost all points on the unit circle, where the values coincide with the values of $R$.

Since the function $R$ is positive on $\BT$, it follows that $T_R$ is a positive operator on $\ell_+^2(\BC^m)$. Under some additional constraints on $T_R$, the positivity of $T_R$ implies that $R$ admits an \emph{outer spectral factor} (see \cite[Proposition V.4.2]{NFBK09}), that is,  there exists a function $\Phi\in H_{r\ts m}^\iy$, for some integer $r\leq m$, such that
\begin{equation}\label{outs}
R= \Phi^* \Phi, \ \ \mbox{i.e.,}\ \  T_R=T_\Phi^* T_\Phi,
\quad\mbox{and}\quad \kr T_\Phi^*=\{0\}.
\end{equation}
The latter condition says that $T_\Phi$ has dense range, i.e., $\Phi$ is outer. The function  $\Phi$ is unique up to a unitary constant matrix on the left, that is, if $\Psi$ is another outer function satisfying $R= \Psi ^* \Psi $, then $\Phi$ and $\Psi$ are matrix functions of the same size, and $\Phi(\cdot) = U\Psi(\cdot)$ where $U$ is a constant unitary  matrix. With some abuse of terminology, we shall refer to $\Phi$ as {\em the} outer spectral factor of $R$. See \cite{NFBK09,RR85} for further details.

We start with a few preliminary results.

\begin{lem}\label{lemNM}
Let $\Phi$ be the $r\ts m$ outer spectral factor of the function  $R$ given by \eqref{defR}. Set $\sN_{\Phi}=\overline{\im H_G + \im H_K}$, and let $\sM_{\Phi}$ be the inverse image of
$\sN_{\Phi}$ under the map $T_\Phi^*$, i.e.,
\begin{equation}
\label{defM}
\sM_{\Phi}=(T_\Phi^*)^{-1}\left[\sN_{\Phi}\right]=\{f\in \ell_+^2(\BC^r)\mid T_\Phi^*f\in  \overline{\im H_G + \im H_K}\,\}.
\end{equation}
Then $\sM_{\Phi}$ is a subspace of $\ell^2_+(\BC^r)$, $\dim \sM_{\Phi}\leq \dim \sN_{\Phi}$, and $\sM_{\Phi}$ is invariant under the backward shift $S_r^*$. Moreover,
\begin{equation}
\label{inM}
\im H_\Phi E_m =\im S_r^*T_\Phi E_m\subset  \sM_{\Phi}.
\end{equation}
\end{lem}

\begin{proof}[\bf Proof.]
Since $T_\Phi^*$ is a continuous linear map, the inverse image of the closed linear manifold  $\sN_{\Phi}$ under $T_\Phi^*$ is  again linear and closed. Thus $\sM_{\Phi}$ is a subspace. The bound on $\dim\sM_{\Phi}$ follows from the injectivity of $T_\Phi^*$. The fact that $S_m^*H_G=H_GS_p$ and $S_m^*H_K=H_KS_q$ implies that
\begin{align*}
S_m^*\Big(\overline{\im H_G+ \im H_K}\Big) &\subset \overline{S_m^*\im H_G +S_m^* \im H_K}\\
&\qquad=\overline{\im H_GS_p + \im H_K S_q} \subset \overline{\im H_G+ \im H_K}.
\end{align*}
Thus $\sN_{\Phi}$ is invariant under $S_m^*$. Take $f\in \sM_{\Phi}$, i.e., $T_\Phi^*f\in \sN_{\Phi}$. Using  $S_rT_\Phi =T_\Phi S_m $ we have
\[
T_\Phi^*S_r^*f=S_m^*T_\Phi^*f\in  S_m^*\sN_{\Phi}\subset \sN_{\Phi}.
\]
Thus $S_r^*f\in \sM_{\Phi}$.  Hence $\sM_{\Phi}$ is invariant under the backward shift $S_r^*$.

Next we prove \eqref{inM}. Inspecting the first columns in $H_\Phi$ and $T_\Phi$ yields $H_\Phi E_m =S_r^*T_\Phi E_m$. Hence the identity in \eqref{inM} holds. Take $u\in \BC^m$, and put $x=S^*T_\Phi E_mu$.
Then
\begin{align*}
T_\Phi^*x &=  T_\Phi^*S_r^*T_\Phi E_mu=S_m^*T_\Phi^*T_\Phi E_mu
=S_m^*T_RE_mu\\
&=S_m^*(T_GT_G^*+H_GH_G^*)E_mu-S_m^*(T_KT_K^*+H_KH_K^*)E_mu.
\end{align*}
Note that $T_G^*E_m=E_p G(0)^*$, $S_m^*T_GE_p=H_GE_p$ and $S_m^* H_G=H_G S_p$.
Hence
\[
S_m^*(T_GT_G^*+H_GH_G^*)E_mu=
H_G(E_pG(0)^*+ S_pH_G^*E)u\in \im H_G.
\]
Similarly, $S_m^*(T_KT_K^*+H_KH_K^*)E_mu\in \im H_K$.
This shows that $T_{\Phi}^*x$ belongs to $\im H_G+\im H_K\subset \sN_{\Phi}$, and thus $x\in\sM_{\Phi}$.  Hence $\im S_r^*T_\Phi E_m\subset\sM_{\Phi}$.
\end{proof}

\begin{cor}\label{C:ImHPhiIncl}
Let $\Phi$ be the $r\ts m$ outer spectral factor of the function  $R$ given by \eqref{defR}. Define $\sM_{\Phi}$ by \eqref{defM}. Then $\im H_\Phi\subset\sM_{\Phi}$.
\end{cor}

\begin{proof}[\bf Proof.]
By \eqref{inM}, we see that the range of the first block column of $H_\Phi$ is in $\sM_{\Phi}$.
Since $H_\Phi S_m=S_r^* H_\Phi$, it follows that $H_\Phi S_m^l=S_r^{*l} H_\Phi$ holds for any positive integer $l$. The fact that $\sM_{\Phi}$ is invariant under $S_r^*$ then shows that for any positive integer $l$
\[
\im  H_\Phi S_m^l E_m=\im S_r^{*l}H_\Phi E_m=S_r^{*l} \im H_\Phi E_m\subset
S_r^{*l}\sM_{\Phi}\subset\sM_{\Phi}.
\]
This shows that the range of each column of $H_\Phi$ is in $\sM_{\Phi}$, and thus the range of  $H_\Phi$ is included in $\sM_{\Phi}$.
\end{proof}

By the Beurling-Lax-Halmos theorem, the fact that the space $\sM_{\Phi}$  is invariant under the backward shift implies $\sM_{\Phi}=\kr T_\tht^*$ for some inner function $\tht\in H_{r\ts k}^\iy$, with $k$ some nonnegative integer, $k\leq r$. This $\tht$ is unique up to a constant unitary matrix from the right. Despite this mild form of non-uniqueness, we shall refer to $\tht$  as {\em the} inner function associated with the space $\sM_{\Phi}$.

\begin{prop}\label{propF}
Let $\Phi$ be the $r\ts m$ outer spectral factor of the function $R$ given by \eqref{defR}, and let $\tht$ be the $r\ts k$ inner function associated with the space $\sM_{\Phi}$ in \eqref{defM}. Then $F=\Phi^*\tht$ belongs to $H_{m\ts k}^\iy$. Moreover, we have
\begin{itemize}
\item[(i)] $T_G T_G^*-T_KT_K^*-T_F T_F^*=T_\Phi^* P_{\sM_{\Phi}} T_\Phi -H_G H_G^*
+H_KH_K^*\geq 0$;

\item[(ii)] $\rank(T_G T_G^*-T_KT_K^*-T_F T_F^*)\leq \dim(\overline{\im H_G+\im H_K})$.

\end{itemize}
Here $P_{\sM_{\Phi}}$ is the orthogonal projection on $\ell_+^2(\BC^k)$ with range $\sM_{\Phi}$.
If in addition $H_G H_G^*-H_KH_K^*\geq 0$, then $\rank(T_G T_G^*-T_KT_K^*-T_F T_F^*)\leq \dim\sM_{\Phi}$.
\end{prop}

\begin{proof}[\bf Proof.]
Since $\Phi^*$ and $\tht$ are matrix-valued $L^\iy$-functions, we have $F\in L^\iy_{m\ts k}$. To see that $F\in H^\iy_{m\ts k}$ it suffices to show that $H_{F,-}=0$. However, this is the same as showing that $H_{F^*,+}=0$. Note that $\ker T_{\tht}^*=\ell^2_+(\BC^{r})\ominus \sM_{G,K}$, by definition of $\tht$. Hence $H_\Phi T_{\tht}=0$, by Corollary \ref{C:ImHPhiIncl}. Thus the third identity in \eqref{UsefulIDs} yields
\[
H_{F^*,+}=H_{\tht^*\Phi,+}=T_{\Phi}^*H_\Phi=0,
\]
and it follows that $F\in H^\iy_{m\ts k}$, as claimed.


Next we deal with item (i). Since $\sM_{\Phi}=\kr T_\tht^*$ and $\tht$ is inner, $\sM_{\Phi}^\perp=\im T_\tht$ and $T_\tht T_\tht^*$ is the orthogonal projection onto $\sM_{\Phi}^\perp$. In particular, $I-P_{\sM_{\Phi}}=T_\tht T_\tht^*$.  Applying the second identity in \eqref{UsefulIDs} yields
\begin{equation}\label{hank44}
T_{R} = (T_G T_G^*-T_KT_K^*) + (H_G H_G^*-H_KH_K^*).
\end{equation}
With   \eqref{hank44} and $I-P_{\sM_{\Phi}}=T_\tht T_\tht^*$ we obtain
\begin{align*}
&(T_G T_G^*-T_KT_K^*)+ (H_G H_G^*- H_KH_K^*) =T_R= T_\Phi^* T_\Phi=\\
    &\hspace{1cm} = T_\Phi^* P_{\sM_{\Phi}} T_\Phi +T_\Phi^*(I-P_{\sM_{\Phi}}) T_\Phi=T_\Phi^* P_{\sM_{\Phi}} T_\Phi+T_\Phi^* T_\tht T_\tht^* T_\Phi=\\
    &\hspace{1cm}=T_\Phi^* P_{\sM_{\Phi}} T_\Phi+T_F T_F^*.
\end{align*}
Here we used that $T_\Phi^* T_\tht=T_{\Phi^*}T_\tht=T_F$. This proves the identity in (i).

To show $T_\Phi^* P_{\sM_{\Phi}} T_\Phi-H_G H_G^*+H_KH_K^*$ is positive and to prove the rank constraint on this operator, we apply
Lemma \ref{lempos}  with the following choices of spaces and operators:
\begin{align*}
&\sV= \ell_+^2(\BC^m),\quad   \sV_1=\sN_{\Phi}, \quad   \sV_2=\sV\ominus \sN_{\Phi} \quad X= H_GH_G^*-H_KH_K^*,\\
& \sW=\ell_+^2(\BC^r), \quad \sW_1=\sM_{\Phi}, \quad \sW_2=\sW\ominus \sM_{\Phi}, \quad Y=T_\Phi^* .
\end{align*}
Here $\sN_{\Phi}$ and $\sM_{\Phi}$ are the spaces defined in Lemma \ref{lemNM}. In particular,
\begin{align*}
&X\sV=\im (H_GH_G^*-H_KH_K^*)\subset \Big(\overline{\im H_G+ \im H_K}\Big) =\sN_{\Phi}=\sV_1,   \\
&\hspace{2cm}Y^{-1}[\sV_1]  =  (T_\Phi^*)^{-1}[\sN_{\Phi}]= \sM_{\Phi}=\sW_1.
\end{align*}
Furthermore, we have
\begin{align*}
YY^*-X   & = T_\Phi^* T_\Phi-H_GH_G^*+H_KH_K^*=T_R-H_G H_G^*+H_KH_K^*\\
    &  =T_G T_G^*-T_KT_K\geq 0.
\end{align*}
Hence $T_\Phi P_{\sM_{\Phi}} T_\Phi^*-H_G H_G^*+H_KH_K^*=YP_{\sW_1} Y^*-X$ is positive
by \eqref{pos2M}, and the rank constraint (ii) follows from Lemma \ref{lempos} as well.

Moreover, note that $H_GH_G^*-H_KH_K^*\geq0$ translates to $X\geq0$. Thus, by the last statement of Lemma \ref{lempos} we find that
\[
\rank(T_\Phi^*P_{\sM_{\Phi}} T_\Phi-H_GH_G^*+H_KH_K^*)=\rank(YP_{\sW_1}Y^*-X)\leq \dim\sW_1,
\]
which, together with $\dim\sW_1=\dim\sM_{\Phi}$, proves the last claim.
\end{proof}

We will now prove the main result of the present paper.

\begin{proof}[\bf Proof of Theorem \ref{T:main}.]
Let $\sM_{\Phi}$ and $\sN_{\Phi}$ be as in Lemma \ref{lemNM}, and define  $F$ as in Proposition \ref{propF}. Thus   $F=\Phi^*\tht$, where $\Phi$ is the $r\ts m$ outer spectral factor of  the function $R$, and   $\tht$ is the $r\ts k$ inner function associated with the space $\sM_{\Phi}$ in \eqref{defM}. We know that $F\in H^\iy_{m\ts k}$.  With this choice of $F$, Proposition \ref{propF} tells us directly that items (i) and (ii) in Theorem \ref{T:main} are fulfilled.
\end{proof}

\section{The case where $G$ and $K$ are rational matrix functions.}
\label{S:RKrat}\setcounter{equation}{0}

Let $G\in\fR H^\iy_{m\ts p}$ and $K\in\fR H^\iy_{m\ts q}$ such that $T_GT_G^*-T_KT_K\geq0$. The aim of this section is to prove Theorem \ref{T:mainrat}. In addition we will derive a criterion for the case that $\rank (T_GT_G^*-T_KT_K^*)<\iy$.

In the previous section we observed that $T_R\geq 0$, where $R\in L^\iy_{m\ts m}$ is given by \eqref{defR}. Since $G$ and $K$ are rational, so is $R$, and this, together with $T_R\geq 0$, implies $R$ admits an outer spectral factor $\Phi\in\fR H^\iy_{r\ts m}$, sor some $r\leq m$, see \cite[Section 6.6]{RR85}. Also note that $\d(G)=\rank H_G<\iy$ and $\d(K)=\rank H_K<\iy$ imply that the subspace $\sN_{\Phi}$ of Lemma \ref{lemNM} is finite dimensional, and hence the subspace $\sM_{\Phi}$ in \eqref{defM} is finite dimensional, since $\ker T_\Phi^*=\{0\}$. Then Theorem 4.3.2 in \cite{FB10} yields that the inner function $\tht$ associated with ${\sM_{\Phi}}$ is a two-sided inner rational matrix function, that is, $\tht\in \fR H^\iy_{r\ts r}$ and $\tht \tht^*=\tht^*\tht$ is identically equal to $I_r$.

The next proposition provides the relations between the McMillan degrees given in Theorem \ref{T:mainrat}.

\begin{prop}\label{P:degrees}
Let $G\in\fR H_{m\ts p}^\iy$ and $K\in\fR H_{m\ts q}^\iy$ with $T_GT_G^*-T_KT_K^*\geq0$. Define ${\sM_{\Phi}}$ as in \eqref{defM}. Then the functions $R$, $\Phi$, $\tht$ and $F$ defined in Section \ref{S:NewApproach} are all rational matrix functions and the following bounds on their McMillan degrees apply:
\begin{equation}\label{degMbounds}
\half\degM(R)=\degM(\Phi)\leq\degM(F)=\degM(\tht)=\dim{\sM_{\Phi}}.
\end{equation}
Moreover, we have $H_\tht H_\tht^*=P_{\sM_{\Phi}}$ and $H_FH_F^*=T_\Phi^* P_{\sM_{\Phi}} T_\Phi$.
\end{prop}

\begin{proof}[\bf Proof.]
Corollary \ref{C:ImHPhiIncl} implies $\im H_\Phi\subset{\sM_{\Phi}}$. Hence
\[
\degM(\Phi)=\rank(H_\Phi)=\dim \im H_\Phi\leq \dim{\sM_{\Phi}}.
\]
Moreover, we have $H_{R,+}=H_{\Phi^*\Phi,+}=T_{\Phi}^* H_\Phi$, by the third identity in \eqref{UsefulIDs} applied to $V^*W=\Phi^*\Phi$. Since $\Phi$ is outer, $\kr T_{\Phi}^*=\{0\}$, and therefore $\rank H_{R,+}=\rank H_\Phi=\degM(\Phi)$. By $T_R\geq 0$, we have $H_{R,-}=H_{R,+}^*$. In particular, $\rank H_{R,-}=\rank H_{R,+}^*=\rank H_{R,+}$, and thus $\degM(R)=2 \rank H_{R,+}=2\degM(\Phi)$.

The fact that $\tht$ is inner with ${\sM_{\Phi}}=\kr T_\tht^*$ implies $T_\tht T_\tht^* =I-P_{\sM_{\Phi}}$. Since $\tht$ is two-sided inner, we have $\tht \tht^*=\tht^*\tht=I_r$, hence $T_{\tht\tht^*}=I$. Now apply the second identity of \eqref{UsefulIDs}. This yields
\[
H_\tht H_\tht^*=T_{\tht \tht^*}-T_\tht T_\tht^*=I-T_\tht T_\tht^*=P_{\sM_{\Phi}}.
\]
Hence
\[
\degM(\tht)=\rank H_\tht=\rank(H_\tht H_\tht^*)=\rank P_{\sM_{\Phi}}=\dim{\sM_{\Phi}}.
\]

Recall that $F=\Phi^*\tht$. Hence, by the third identity of \eqref{UsefulIDs}, we obtain that $H_F=H_{\Phi^*\tht}=T_{\Phi}^*H_\tht$. Since $\Phi$ is outer, we have $\kr T_{\Phi}^*=\{0\}$, which implies $\degM(F)=\rank H_F=\rank(T_{\Phi}^* H_\tht)=\rank H_\tht=\degM(\tht)$. Finally, $H_F=T_{\Phi}^*H_\tht$ together with $H_\tht H_\tht^*=P_{\sM_{\Phi}}$ implies $H_FH_F^*=T_{\Phi}^*P_{\sM_{\Phi}} T_\Phi$.
\end{proof}

Note that $\dim (\overline{\im H_G + \im H_K})\leq \d(G)+\d(K)<\iy$, since $G\in \fR H^\iy_{m\ts p}$ and $K\in\fR H^\iy_{m\ts q}$. Hence, replacing $K$ by $\wtil{K}=[\, K\ F\,]$, reduces the original Leech equation \eqref{Leech1} to one where
\begin{equation}
\label{rankcond}
\rank (T_GT_G^*-T_KT_K^*)<\iy.
\end{equation}

We will next focus on the case of the Leech equation where the rank constraint \eqref{rankcond} holds. The following theorem provides necessary and sufficient conditions for \eqref{rankcond} to hold.

\begin{thm}\label{T:rational}
Let $G\in\fR H^\iy_{m\ts p}$ and $K\in\fR H^\iy_{m\ts q}$ with $T_GT_G^*-T_KT_K^*\geq0$. Define $R\in \fR L^\iy_{m\ts m}$ by \eqref{defR}. Then the following statements are equivalent:
\begin{itemize}
\item[\textup{(i)}] $\rank (T_GT_G^*-T_KT_K^*)<\iy$;
\item[\textup{(ii)}] $T_R=0$;
\item[\textup{(iii)}] $G(e^{it})G(e^{it})^*=K(e^{it})K(e^{it})^*\quad (t\in[0,2\pi])$.
\end{itemize}
Moreover, in this case
\begin{equation}\label{TvsH2}
T_GT_G^*-T_KT_K^*=H_KH_K^*-H_GH_G^*
\end{equation}
and
\begin{equation}\label{RankBound}
\degM(G)\leq\degM(K),\quad
\degM(K)-\degM(G)\leq \rank(T_GT_G^*-T_KT_K^*)\leq \degM(K).
\end{equation}
Here $\degM(G)$ and $\degM(K)$ denote the McMillan degrees of $G$ and $K$, respectively.
\end{thm}

\begin{proof}[\bf Proof.]
Note that (iii) is equivalent to $R(e^{it})=0$ for each $t\in [0,2\pi]$, hence to $T_R=0$, since $R\in \fR L^\iy_{m\ts m}$. Thus (ii) $\Leftrightarrow$ (iii).

The fact that $G$ and $K$ are rational matrix $H^\iy$-functions implies that $H_G$ and $H_K$ have finite rank, and thus $\rank(H_GH_G^*-H_KH_K^*)<\iy$. From formula \eqref{hank44} it then follows that (i) holds if and only if $\rank T_R<\iy$. However, $R$ is a rational matrix function with no poles of the circle, and   thus continuous on the circle. This implies that $\rank T_R<\iy$ holds if and only if $R(e^{it})=0$ for all $t\in [0,2\pi]$, and thus $T_R=0$. Hence (i) $\Leftrightarrow$ (ii).

The combination of $T_R=0$ and formula \eqref{hank44} gives \eqref{TvsH2}.

Note that for any positive Hilbert space operators $Z$ and $Y$ on $\sV$, the inequality $Z\geq Y$ implies $\rank Z\geq \rank Y$. Indeed, by Douglas' Factorization Lemma there exists a contraction $Q$ on $\sV$ such that $Y^\half=Q Z^\half$. Hence
\[
\rank Y=\rank Y^\half=\rank (Q Z^\half)\leq \rank (Z^\half)=\rank(Z).
\]
Applying this inequality with $Z=H_KH_K^*$ and $Y=H_GH_G^*$ and noting that $Z-Y=H_KH_K^*-H_GH_G^*=T_GT_G^*-T_KT_K^*\geq 0$, we obtain
\[
\d(K)=\rank (H_KH_K^*)
=\rank(Z)\geq\rank(Y)=\rank(H_GH_G)=\d(G).
\]
If we take $Z=H_KH_K^*$ and $Y=H_KH_K^*-H_GH_G^*=T_GT_G^*-T_KT_K^*$, then clearly $Z\geq Y$, and thus
\[
\d(K)=\rank(H_KH_K^*)=\rank(Z)\geq \rank (Y)=\rank (T_GT_G^*-T_KT_K^*).
\]
In addition to $Z$ and $Y$, set $V=H_GH_G^*$. Then $Z=Y+V$ implies
\[
\rank(Z)=\rank(Y+V)\leq \rank(Y)+\rank(V).
\]
Since $\d(G)=\rank(V)$ and $\d(K)=\rank(Z)$, the last part of \eqref{RankBound} holds.
\end{proof}

%
%

\begin{remark}
If the matrix $H^\iy$-functions $G$ and $K$ are continuous, then the first part of Theorem \ref{T:rational} goes through in a slightly altered form. One only has to replace (i) by: $T_GT_G^*-T_KT_K^*$ is compact. The argumentation is similar to the one given in the proof of Theorem \ref{T:rational}, where we now use that $H_G$ and $H_K$ are compact, since $G$ and $K$ are continuous, and that $R$ being continuous together with $T_R$ compact implies $T_R=0$, and hence $R=0$.
\end{remark}

How restrictive condition \eqref{rankcond} can be becomes evident when considering the Toeplitz corona problem.

\begin{cor}\label{C:corona1}
Let $G\in \fR H^\iy_{m\ts p}$ such that  $T_GT_G^*\geq I$, i.e., \eqref{poscond1} holds with $K(z)=I_m$ for each $z\in\BD$. Then $\rank (T_GT_G^*-I)<\iy$ holds if and only if $G$ is a constant matrix function whose value is a co-isometry.
\end{cor}

\begin{proof}[\bf Proof.]
Clearly if $G$ is a constant matrix function whose value is a co-isometry, then $T_GT_G^*=I$, and hence $T_GT_G^*-I$ has finite rank.

Conversely, assume $T_GT_G^*-I$ has finite rank. By Theorem \ref{T:rational},  $R=GG^*-I_m=0$. Thus $GG^*=I_m$. In particular, the values of $G$ are co-isometries. By the second identity in \eqref{UsefulIDs} we have $I_{\ell^2_+(\BC^m)}=T_{GG^*}=T_GT_G^*+H_GH_G^*$. Thus $-H_GH_G^*=T_GT_G^*-I\geq0$. This can only occur if $H_G=0$, i.e., if $G$ is constant matrix function.
\end{proof}

\begin{cor}\label{C:leftfact}
Let $G\in\fR H^\iy_{m\ts p}$ and $K\in\fR H^\iy_{m\ts q}$ with $T_GT_G^*-T_KT_K^*\geq0$. Define $R$, $\Phi$, $\tht$, and $F$ as in Section \ref{S:NewApproach}. Then
\begin{equation}\label{factors}
\Phi^*\Phi=R=FF^*\ands \Phi=\tht F^*.
\end{equation}
Moreover, $T_R>0$ if and only $\Phi$ is invertible outer, that is, $r=m$ and $\Phi$ has an inverse in $H_{m\ts m}^\iy$. In this case $F$ is invertible in $L^\iy_{m\ts m}$ with an anti-analytic inverse.
\end{cor}

The first two identities in \eqref{factors} say that $\Phi$ is a right and $F$ a left spectral factors of $R$. The last identity, together with $\tht$ two-sided inner, provides a Douglas-Shapiro-Shields factorization of $\Phi$, cf., \cite[Chapter 4]{FB10}.

\begin{proof}[\bf Proof of Corollary \ref{C:leftfact}.]
The identity $\Phi^*\Phi=R$ holds by definition of $\Phi$. Applying Theorem \ref{T:rational} with $K$ replaced by $\wtil{K}=[\,K\ F\,]$, where we note that condition (i) is satisfied by Theorem \ref{T:main}, yields
\[
GG^*=\wtil{K}\wtil{K}^*=KK^*+FF^*,\quad\mbox{i.e.}\quad
FF^*=GG^*-KK^*=R.
\]
Recall that $F$ is defined as $F=\Phi^*\tht$. Hence $F^*=\tht^*\Phi$. Since $\tht$ is two-sided inner, $\tht\tht^*$ is identically equal to $I_r$. Hence $\Phi=\tht F^*$.

It is well known that $T_R>0$ holds if and only if its outer spectral factor is invertible outer, c.f., \cite[Proposition 10.2.1]{FB10}. Assume $T_R>0$. Then $\Phi$ and $\tht$ are invertible with $\Phi^{-1}\in H^\iy_{m\ts m}$ and $\tht^{-1}=\tht^*$. This shows that $F=\Phi^*\tht$ is invertible in $L^\iy_{m\ts m}$, with inverse $(\Phi^*\tht)^{-1}=\tht^*(\Phi^{*})^{-1}=\tht^*(\Phi^{-1})^{*}$. Since $\tht^*$ and $(\Phi^{-1})^{*}$ are both anti-analytic, so is $F^{-1}$.
\end{proof}

\begin{proof}[\bf Proof of Theorem \ref{T:mainrat}.]
We observed at the beginning of the present section that $R$ admits an outer spectral factor and that $\tht$ is two-sided inner. The relations between the McMillan degrees of $R$, $\Phi$, $\tht$ and $F$ in \eqref{MMbounds} follow from Proposition \ref{P:degrees}. The identity \eqref{TvsH} follows by replacing $K$ in \eqref{TvsH2} by $\wtil{K}=[\, K\ F\,]$, noting that $\rank(T_GT_G^*-T_{\wtil{K}}T_{\wtil{K}}^*) =\rank(T_GT_G^*-T_KT_K^*-T_FT_F^*)<\iy$ by Theorem \ref{T:main}, and the identities
$T_{\wtil{K}}T_{\wtil{K}}^*=T_KT_K^*+T_FT_F^*$ and $H_{\wtil{K}}H_{\wtil{K}}^*=H_KT_K^*+H_FT_F^*$.
\end{proof}

In case \eqref{rankcond} holds, the following proposition shows how the partial isometry $M_\circ$ in \eqref{contracon2intro} can be computed.

\begin{prop}\label{P:integralform}
Let $G\in\fR H^\iy_{m\ts p}$ and $K\in\fR H^\iy_{m\ts q}$ such that \eqref{rankcond} holds. Let $\nu=\rank(T_GT_G^*-T_KT_K^*)<\iy$. Then $\nu\leq \degM(K)$, the space $\sH_0$ in \eqref{defMintro} can be taken to be $\BC^\nu$, and in that case the partial isometry $M_\circ$ in  \eqref{defMintro} and \eqref{contracon2intro} can be computed via $M_\circ=M_1^+ M_*$ with
\begin{align*}
M_1=
\frac{1}{2\pi} \int_0^{2\pi} V(e^{i \omega})^*V(e^{i \omega}) d \omega, \quad&\quad
M_* =
\frac{1}{2\pi} \int_0^{2\pi} V(e^{i \omega})^*W(e^{i \omega}) d \omega,\\
 V(e^{it})=\mat{cc}{e^{it}\hat\la_\circ(e^{it})&G(e^{it})}, \ &\
W(e^{it})=\mat{cc}{\hat\la_\circ(e^{it})&K(e^{it})},\quad \mbox{a.e.}
\end{align*}
and $M_1^+$ the Moore-Penrose pseudo-inverse of $M_1$.
\end{prop}

\begin{proof}[\bf Proof.]
Recall from the introduction that  $\dim\sH_\circ=\rank(T_GT_G^*-T_KT_K^*)=\nu$. Since $\nu<\iy$, we can apply a linear transformation identifying $\sH_\circ$ with $\BC^\nu$, and since $\sH_\circ$ comes from the factorization of $T_GT_G^*-T_KT_K^*$, we can just as well apply this transformation and take $\sH_0$ to be $\BC^\nu$. The bound on $\nu$ is a direct consequence of \eqref{RankBound}.

The formula for $M_\circ$ follows by applying Lemma \ref{lemrep} with the given choice of $V$ and $W$. Note that the identity \eqref{MNcondition} follows from \eqref{fundid1}. The square summability of the Taylor coefficients of $V$ and $W$ follows from the boundedness of $T_G E_p$, $T_K E_q$ and $\la_\circ$ (as defined in the introduction), as operators mapping into $\ell^2_+(\BC^m)$. Hence all conditions are satisfied, and Lemma \ref{lemrep} applies.
\end{proof}

We conclude this section with two examples.

\begin{example}
According to Proposition \ref{propF}, if $H_GH_G^*-H_KH_K^*\geq0$, then the upper bound on the rank of $T_GT_G^*-T_KT_K^*-T_FT_F^*$ in item (i) can be improved to $\dim\sM$, with $\sM$ as defined in Lemma \ref{lemNM}. This improvement can be arbitrarily large. Let $l$ be a positive integer, take for $G$ any rational function of McMillan degree $l$ and take $K=G$. Clearly $R=GG^*-KK^*=0$, thus $\Phi=0$, which implies $\sM=\{0\}$. Hence $\dim\sM=0$; a solution with McMillan degree 0 is obviously $X(z)=1$, $z\in\BC$.      On the other hand $\dim(\overline{\im H_G + \im H_K })=\dim(\im H_G)=\degM(G)=l$. Hence we have an improvement of $l$.
\end{example}

\begin{example}\label{E:poly}
Let $G$ and $K$ are matrix polynomials whose values are matrices of size $m\ts p$, respectively $m\ts q$, say with degrees $d_1$, respectively $d_2$. Assume that the last coefficients of $G$ and $K$, i.e, corresponding to $z^{d_1}$ and $z^{d_2}$, have full rank and that $p,q\geq m$. This implies that the last coefficients of $G$ and $K$ admit a right inverse. Note that $H_G$ and $H_K$ only have entries on the first $d_1$, respectively $d_2$, anti-diagonals, starting in the left upper corner. Since the last coefficients of $G$ and $K$ admit a left inverse, it follows that  $\degM(G)=\rank H_G=m\cdot d_1$ and $\degM(K)=\rank H_K=m\cdot d_2$. Now also assume that $T_GT_G^*-T_KT_K^*\geq 0$. Applying Theorem \ref{T:main}, and following the subsequent procedure we obtain that there exists a rational matrix solution $X$ to \eqref{Leech1}. The McMillan degree of $X$ is bounded by $\degM(G)+\degM(K)=m(d_1+d_2)$. However, in this case the rank constraint in item (ii) of Theorem \ref{T:main} gives a much sharper bound, namely $\dim(\im H_G + \im H_K)\leq m\, \max\{d_1,d_2\}$, due to the specific structure of $H_G$ and $H_K$. Note that this bound is  in line with \cite{Trent13} (where the factor $m$ does not appear, but should be there).
\end{example}

\setcounter{equation}{0}
\section{Appendix}\label{S:appendix}

In this appendix we prove two results of a general operator theoretical nature that are used in the paper.

\begin{lem} \label{lempos}
Let  $\sV=\sV_1\oplus \sV_2$ and  $\sW=\sW_1 \oplus \sW_2 $ be  Hilbert space direct sums, and  let $X:\sV\to \sV$ and $Y:\sW\to \sV$ be operators. Assume that  $X$   is selfadjoint and $X \sV\subset \sV_1$, and that $\sW_1=Y^{-1}[\sV_1]$, i.e.,  $\sW_1$ is the inverse image of  $\sV_1$ under $Y$.
Finally, let $P_{\sW_1}$ be the orthogonal projection of  $\sW$ onto $\sW_1$. Then
\begin{equation}
\label{pos2M}
YY^*-X \geq 0\  \Longleftrightarrow\  YP_{\sW_1} Y^*-X \geq 0.
\end{equation}
Moreover, $\rank (YP_{\sW_1} Y^*-X)\leq \dim \sV_1$. Assume $YY^*-X \geq 0$ and
in addition that $Y$ is injective and $X\geq 0$. Then $\rank(YP_{\sW_1} Y^*)=\dim\sW_1$, $\rank X\leq \dim \sW_1$ and $\rank (YP_{\sW_1} Y^*-X)\leq \dim \sW_1$.
\end{lem}

\begin{proof}[\bf Proof.]
Using the decompositions  $\sV=\sV_1\oplus \sV_2$ and  $\sW=\sW_1 \oplus \sW_2 $ we represent   $X$ and $Y$ as $2\ts 2$ operator matrices, as follows:
\begin{equation}
\label{partXY}
X=\begin{bmatrix}  X_1 &0 \\  0 &0 \end{bmatrix}:
\begin{bmatrix} \sV_1 \\ \sV_2  \end{bmatrix} \to
\begin{bmatrix} \sV_1 \\  \sV_2  \end{bmatrix}, \
Y=\begin{bmatrix}  Y_1 & Y_2 \\  0  & Y_3 \end{bmatrix}:
\begin{bmatrix} \sW_1 \\  \sW_2  \end{bmatrix} \to
\begin{bmatrix} \sV_1 \\  \sV_2 \end{bmatrix}.
\end{equation}
Note that the  zeros in the operator matrix for $X$ follow from the fact that $X$   is selfadjoint and $X \sV\subset \sV_1$.  The zero in the left lower corner of the operator matrix for $Y$ is a consequence of  $\sW_1=Y^{-1}[\sV_1]$. Indeed, the latter equality implies that $Y$ maps  $\sW_1$ into  $\sV_1$.  The identity $\sW_1=Y^{-1}[\sV_1]$ also  implies that $Y_3$ is one-to-one.  To see this, assume $Y_3u=0$ for some $u\in \sW_2$. Then $Yu\in \sV_1$. But the latter can only happen when $u\in Y^{-1}[\sV_1]=\sW_1$. Thus $u\in \sW_1\cap \sW_2$, and hence $u=0$. Therefore, $Y_3$ is one-to-one.

Next, observe that the partitionings in  \eqref{partXY} imply that
\begin{align}
&YY^*-X=  \begin{bmatrix}  I_{\sV_1} & Y_2 \\  0  & Y_3 \end{bmatrix}
\begin{bmatrix}  Y_1Y_1^*-X_1  & 0 \\  0  & I_{\sW_2} \end{bmatrix}
\begin{bmatrix}  I_{\sV_1} & 0 \\  Y_2^*  & Y_3^* \end{bmatrix}\mbox{ on }\mat{c}{\sV_1\\\sV_2}, \label{idpos1}\\[.2cm]
&\hspace{1.5cm} YP_{\sW_1} Y^*-X =  \begin{bmatrix}  Y_1Y_1^*-X_1 & 0 \\  0  & 0\end{bmatrix}\mbox{ on }\mat{c}{\sV_1\\\sV_2}.\label{idpos2}
\end{align}

Now assume that the inequality in the right hand side of \eqref{pos2M} holds. This implies that  the operator matrix  in the right hand side of  \eqref{idpos2} is positive. But then the same holds true for the operator defined by the second operator matrix in the right hand side of \eqref{idpos1}. The equality \eqref{idpos1} then  shows that $YY^*-X $ is a positive operator, and  the implication $\Longleftarrow$ in \eqref{pos2M} is proved.

To prove the reverse implication assume that $YY^*-X $ is a positive operator.
Since $Y_3$ is one-to-one, the operator $U$ from $\sV_1\oplus \sV_2$ to $\sV_1\oplus \sW_2$ defined by the third operator matrix in the right hand side of  \eqref{idpos1} has a dense range. Using \eqref{idpos2} and the positivity of $YY^*-X$, we see that
\[
\big\lg \begin{bmatrix}  Y_1Y_1^*-X_1 & 0 \\  0  & I_{\sW_2} \end{bmatrix}
Uv, Uv \,\big\rg\geq 0\quad  \mbox{for all $v\in \sV=\sV_1\oplus \sV_2$}.
\]
But the range of $U$ is dense. Hence, by continuity,  we get
\[
\big\lg \begin{bmatrix}  Y_1Y_1^*-X_1 & 0 \\  0  & I_{\sM^\perp} \end{bmatrix}y, y\,\big\rg\geq 0\quad  \mbox{for all $y\in \sV_1\oplus \sW_2$}.
\]
It follows that $ Y_1Y_1^*-X_1$ is positive, and by \eqref{idpos2}
the same holds true for the operator $YP_{\sW_1} Y^*-X $. This proves the implication $\Longrightarrow$ in \eqref{pos2M}.

The decomposition \eqref{idpos2} shows clearly that $\rank(YP_{\sW_1} Y^*-X)\leq\dim\sV_1$.

Note that if $Y$ is injective, we have $\rank(YP_{\sW_1}Y^*)=\rank(P_{\sW_1}) =\dim\sW_1$. Assuming $YY^*-X\geq0$, we have $YP_{\sW_1}Y^*\geq X$. By Douglas' Factorization Lemma, $X^\half= K P_{\sW_1}Y^*$ for some contraction $K$, and hence
\[
\rank X=\rank X^\half=\rank ( K P_{\sW_1}Y^*)\leq \rank(P_{\sW_1}Y^*)
=\rank(YP_{\sW_1}Y^*).
\]
Thus $\rank X\leq \dim\sW_1$. A similar argument applied to $YP_{\sW_1}Y^*\geq YP_{\sW_1}Y^*-X$ shows $\rank(YP_{\sW_1}Y^*-X)\leq \rank(YP_{\sW_1}Y^*)=\dim\sW_1$.
\end{proof}

\begin{lem}\label{lemrep}
Consider two matrix functions $V$ and $W$, analytic on $\BD$,  with valued $V(z):\BC^k\to\BC^p$ and $W(z):\BC^\nu\to\BC^p$, $z\in\BD$, and Taylor expansions $V(z) = \sum_{j=0}^\infty z^j V_j$ and $W(z) = \sum_{j=0}^\infty z^j W_j$. Assume $\sum_{j=0}^\infty V_j^*V_j<\iy$ and $\sum_{j=0}^\infty W_j^*W_j<\iy$.
If
\begin{equation}\label{MNcondition}
V(z)V(w)^* = W(z)W(w)^*
\quad \mbox{for all }  z,w \in \mathbb{D},
\end{equation}
then there exists a partial isometry $M:\BC^\nu\to\BC^k$ such that $V (z)M = W(z)$ for all $z$ in $\mathbb{D}$. Moreover, this partial isometry $M$ is given by $M =M_1^+M_*$ with
\begin{align}\label{Ucom}
M_1 &=
\frac{1}{2\pi} \int_0^{2\pi} V(e^{i \omega})^*V(e^{i \omega}) d \omega
= \sum_{j=0}^\infty V_j^* V_j \nonumber \\
M_* &=
\frac{1}{2\pi} \int_0^{2\pi} V(e^{i \omega})^*W(e^{i \omega}) d \omega
=\sum_{j=0}^\infty V_j^* W_j.
\end{align}
Here $M_1^+$ denotes the Moore-Penrose pseudo inverse of $M_1$.
\end{lem}

\begin{proof}[\bf Proof.]
The assumption yields we can define operators $\Omega_1$ and $\Omega_2$ by
\[
\Omega_1 = \begin{bmatrix}
             V_0 \\
             V_1 \\
             V_2 \\
             \vdots \\
           \end{bmatrix}:\mathbb{C}^k \rightarrow \ell_+^2(\mathbb{C}^p)
\quad \mbox{and}\quad
\Omega_2 = \begin{bmatrix}
             W_0 \\
             W_1 \\
             W_2 \\
             \vdots \\
           \end{bmatrix}:\mathbb{C}^\nu \rightarrow \ell_+^2(\mathbb{C}^p).
\]
For each $z\in\BD$ we write $\fF_{z}$ for the \emph{point evaluation operator}
\begin{equation*}
\fF_{z}=E_p^*(I-zS_p^*)^{-1}:\ell^2_+(\BC^p)\to \BC^p, \quad \mbox{i.e.,}\quad
\fF_{z}(x_0,x_1,x_2,\ldots)=\sum_{j=0}^\iy z^jx_j.
\end{equation*}
Note that $V(z)=\mathfrak{F}_{p,z}\om_1$ and $W(z)=\mathfrak{F}_{p,z}\om_2$, $z\in\BD$. Hence
\[
\mathfrak{F}_{p,z}(\om_1\om_1^*-\om_2\om_2^*)\mathfrak{F}_{p,w}^* =V(z)V(w)^*-W(z)W(w)^*=0 \quad (z,w\in\BD).
\]
Since $\cap_{z\in\BD}\kr \mathfrak{F}_{p,z}=\{0\}$, it follows that $\om_1\om_1^*=\om_2\om_2^*$. By Douglas' factorization lemma there exists a unique partial isometry $M:\BC^\nu\to\BC^k$ that satisfies $\om_1M=\om_2$ and has $\im\om_2^*$ as initial space and $\im \om_1^*$ as final space. Multiplying both sides with $\mathfrak{F}_{p,z}$ yields $V(z)M=W(z)$, $z\in\BD$. Note that the Moore-Penrose pseudo inverse of $\om_1$ is given by $\om_1^+=(\om_1^*\om_1)^+\om_1^*$. Then $\om_1^+\om_1$ is the orthogonal projection on $\im \om_1^*$. Thus  $M=\om_1^+\om_1M=\om_1^+\om_2 =(\om_1^*\om_1)^+\om_1^*\om_2$. Note that $M_1=\om_1^*\om_1$ and $M_*=\om_1^*\om_2$. Hence $M= M_1^+ M_*$.
\end{proof}

\paragraph{\bf Acknowledgement}
The author thanks Art Frazho and Rien Kaashoek for the useful discussions and their constructive suggestions during the preparation of this paper.

\end{document}